\documentclass[11pt,letter]{article}
\usepackage{amsfonts,epsf,amsmath,amssymb,graphicx,tabularx,booktabs,url}

\newtheorem{theorem}{\bf Theorem}[section]
\newtheorem{corollary}[theorem]{\bf Corollary}
\newtheorem{lemma}[theorem]{\bf Lemma}

\newtheorem{conjecture}[theorem]{\bf Conjecture}
\newtheorem{remark}[theorem]{\bf Remark}

\newtheorem{definition}[theorem]{\bf Definition}

\newcommand{\proof}{\noindent{\bf Proof.\ }}
\newcommand{\qed}{\hfill $\blacksquare$ \bigskip}

\begin{document}

\title{ \bf On the extremal properties \\ of the average eccentricity}

\author{
Aleksandar Ili\' c \footnotemark[3] \\
Faculty of Sciences and Mathematics, Vi\v segradska 33, 18 000 Ni\v s \\
University of Ni\v s, Serbia \\
e-mail: \tt{aleksandari@gmail.com} \\
}

\date{\today}

\maketitle

\begin{abstract}
The eccentricity of a vertex is the maximum distance from it to another vertex and the average
eccentricity $ecc (G)$ of a graph $G$ is the mean value of eccentricities of all vertices of~$G$.
The average eccentricity is deeply connected with a topological descriptor called the eccentric
connectivity index, defined as a sum of products of vertex degrees and eccentricities. In this
paper we analyze extremal properties of the average eccentricity, introducing two graph
transformations that increase or decrease $ecc (G)$. Furthermore, we resolve four conjectures,
obtained by the system AutoGraphiX, about the average eccentricity and other graph parameters (the
clique number, the Randi\' c index and the independence number), refute one AutoGraphiX conjecture
about the average eccentricity and the minimum vertex degree and correct one AutoGraphiX conjecture
about the domination number.
\end{abstract}

 \vskip 0.1cm

 \vskip 0.1cm

\section{Introduction}

Let $G = (V, E)$ be a connected simple graph with $n = |V|$ vertices and $m = |E|$ edges. Let $deg
(v)$ denotes the degree of the vertex $v$. Let $\delta = \delta (G)$ be the minimum vertex degree,
and $\Delta = \Delta (G)$ be the maximum vertex degree of a graph $G$.

For vertices $u, v \in V$, the distance $d (u, v)$ is defined as the length of the shortest path
between $u$ and $v$ in $G$. The eccentricity of a vertex is the maximum distance from it to any
other vertex,
$$
\varepsilon (v) = \max_{u \in V} d (u, v).
$$

The radius of a graph $r (G)$ is the minimum eccentricity of any vertex. The diameter of a graph $d
(G)$ is the maximum eccentricity of any vertex in the graph, or the greatest distance between any
pair of vertices. For an arbitrary vertex $v \in V$ it holds that $r (G) \leq \varepsilon (v) \leq
d (G)$. A vertex $c$ of $G$ is called central if $\varepsilon (c) = r (G)$. The center $C(G)$ is
the set of all central vertices in $G$. An eccentric vertex of a vertex $v$ is a vertex farthest
away from~$v$. Every tree has exactly one or two center vertices \cite{BuHa90}.

The average eccentricity of a graph $G$ is the mean value of eccentricities of vertices of~$G$,
$$
ecc (G) = \frac{1}{n} \sum_{v \in V} \varepsilon (v).
$$

For example, we have the following formulas for the average eccentricity of the complete graph
$K_n$, complete bipartite graph $K_{n, m}$, hypercube $H_n$, path $P_n$, cycle $C_n$ and star
$S_n$,
$$
ecc (K_n) = 1 \qquad ecc (K_{n, m}) = 2 \qquad ecc (Q_n) = n
$$
$$
ecc (P_n) = \frac{1}{n} \left \lfloor \frac{3}{4}n^2 - \frac{1}{2}n \right \rfloor \qquad ecc (C_n)
= \left \lfloor \frac{n}{2} \right \rfloor \qquad ecc (S_n) = 2 - \frac{1}{n}.
$$

Dankelmann, Goddard and Swart \cite{DaGoSw04} presented some upper bounds and formulas for the
average eccentricity regarding the diameter and the minimum vertex degree. Furthermore, they
examine the change in the average eccentricity when a graph is replaced by a spanning subgraph, in
particular the two extreme cases: taking a spanning tree and removing one edge. Dankelmann and
Entringer \cite{DaEn00} studied the average distance of $G$ within various classes of graphs.

In theoretical chemistry molecular structure descriptors (also called topological indices) are used
for modeling physico-chemical, pharmacologic, toxicologic, biological and other properties of
chemical compounds \cite{GuPo86}. There exist several types of such indices, especially those based
on vertex and edge distances \cite{IlKlMi10,KhYoAsWa09}. Arguably the best known of these indices
is the Wiener index $W$, defined as the sum of distances between all pairs of vertices of the
molecular graph~\cite{DoEnGu01}
$$
W (G) = \sum_{u, v \in V (G)} d (u, v).
$$
Besides of use in chemistry, it was independently studied due to its relevance in social science,
architecture, and graph theory.

Sharma, Goswami and Madan \cite{ShGoMa97} introduced a distance--based molecular structure
descriptor, the eccentric connectivity index, which is defined as
$$
\xi^c = \xi^c (G) = \sum_{v \in V (G)} deg (v) \cdot \varepsilon (v).
$$

The eccentric connectivity index is deeply connected to the average eccentricity, but for each
vertex $v$, $\xi^c (G)$ takes one local property (vertex degree) and one global property (vertex
eccentricity) into account. For $k$-regular graph $G$, we have $\xi^c (G) = k \cdot n \cdot ecc
(G)$.

The index $\xi^c$ was successfully used for mathematical models of biological activities of diverse
nature. The eccentric connectivity index has been shown to give a high degree of predictability of
pharmaceutical properties, and provide leads for the development of safe and potent anti-HIV
compounds. The investigation of its mathematical properties started only recently, and has so far
resulted in determining the extremal values and the extremal graphs \cite{IlGu10,ZhDu09}, and also
in a number of explicit formulae for the eccentric connectivity index of several classes of graphs
\cite{DoSa09} (for a recent survey see \cite{Il10}).

AutoGraphiX (AGX) computer system was developed by GERAD group from Montr\' eal
\cite{Ao06,AGX,CaHa00}. AGX is an interactive software designed to help finding conjectures in
graph theory. It uses the Variable Neighborhood Search metaheuristic (Hansen and Mladenovi\' c
\cite{HaMl01,HaMlMo10}) and data analysis methods to find extremal graphs with respect to one or
more invariants. Recently there is vast research regarding AGX conjectures and series of papers on
various graph invariants: average distance \cite{AuHa07}, independence number \cite{AoBrHa08},
proximity and remoteness \cite{AuHa10}, largest eigenvalue of adjacency and Laplacian matrix
\cite{AuHa10a}, connectivity index \cite{CaGuHaPa03}, Randi\' c index \cite{HaVu09}, connectivity
and distance measures \cite{SeVuAoHa07}, etc. In this paper we continue this work and resolve other
conjectures from the thesis \cite{Ao06}, available online at \url{http://www.gerad.ca/~agx/}.

Recall that the vertex connectivity $\nu$ of $G$ is the smallest number of vertices whose removal
disconnects $G$ and the edge connectivity $\kappa$ of $G$ is the smallest number of edges whose
removal disconnects $G$. Sedlar, Vuki\v cevi\' c and Hansen \cite{SeVuHa07} studied the lower and
upper bounds of $ecc - \delta$, $ecc + \delta$ and $ecc / \delta$, the lower bound for $ecc \cdot
\delta$, and similar relations by replacing $\delta$ with $\nu$ and $\kappa$.

The paper is organized as follows. In Section 2 we introduce a simple graph transformation that
increases the average eccentricity and characterize the extremal tree with maximum average
eccentricity among trees on $n$ vertices with given maximum vertex degree. In Section 3 we resolve
a conjecture about the upper bound of the sum $ecc + \alpha$, where $\alpha$ is the independence
number. In Section~4, we resolve two conjectures about the extremal values $ecc + Ra$ and $ecc
\cdot Ra$, where $Ra$ denotes the Randi\' c index of $G$. In Section 4, we characterize the
extremal graph having maximum value of average eccentricity in the class of $n$-vertex graphs with
given clique number $\omega$. In Section 5, we refute a conjecture about the maximum value of the
product $ecc \cdot \delta$. We close the paper in Section 6 by restating some other AGX conjecture
for the future research and correcting a conjecture about $ecc + \gamma$, where $\gamma$ denotes
the domination number.

\section{The average eccentricity of trees with given maximum degree}

\begin{theorem}
\label{thm-pi} Let $w$ be a vertex of a nontrivial connected graph $G$. For nonnegative integers
$p$ and $q$, let $G (p, q)$ denote the graph obtained from $G$ by attaching to vertex $w$ pendent
paths $P = w v_1 v_2 \ldots v_p$ and $Q = w u_1 u_2 \dots u_q$ of lengths $p$ and $q$,
respectively. If $p \geq q \geq 1$, then
$$
ecc (G (p, q)) < ecc (G (p + 1, q - 1)).
$$
\end{theorem}

\proof Since after this transformation the longer path has increased and the eccentricities of
vertices from $G$ are either the same or increased by one. We will consider three simple cases
based on the longest path from the vertex $w$ in the graph $G$. Denote with $\varepsilon' (v)$ the
eccentricity of vertex $v$ in $G (p + 1, q - 1)$.
\medskip

\noindent {\it Case 1. } The length of the longest path from the vertex $w$ in $G$ is greater than
$p$. This means that the vertex of $G$, most distant from $w$ is the most distant vertex for all
vertices of $P$ and $Q$. It follows that $\varepsilon' (v) = \varepsilon (v)$ for all vertices $w,
v_1, v_2, \ldots, v_p, u_1, u_2, \ldots, u_{q - 1}$, while the eccentricity of $u_q$ increased by
$p + 1 - q$. Therefore,
$$
ecc (G (p + 1, q - 1)) - ecc (G (p, q)) = p + 1 - q > 0.
$$

\noindent {\it Case 2. } The length of the longest path from the vertex $w$ in $G$ is less than or
equal to $p$ and greater than $q$. This means that either the vertex of $G$ that is most distant
from $w$ or the vertex $v_p$ is the most distant vertex for all vertices of $P$, while for the
vertices $w, u_1, u_2, \ldots, u_q$ the most distant vertex is $v_p$. It follows that $\varepsilon'
(v) = \varepsilon (v)$ for vertices $v_1, v_2, \ldots, v_p$, while $\varepsilon' (v) = \varepsilon
(v) + 1$ for vertices $w, u_1, u_2, \ldots, u_{q - 1}$. Also the eccentricity of $u_q$ increased by
at least $1$, and consecutively
$$
ecc (G (p + 1, q - 1)) - ecc (G (p, q)) \geq q + 1 > 0.
$$

\noindent {\it Case 3. } The length of the longest path from the vertex $w$ in $G$ is less than or
equal to $q$. This means that the pendent vertex most distant from the vertices of $P$ and $Q$ is
either $v_p$ or $u_q$, depending on the position. For each vertex from $G$, the eccentricity
increased by~1. Using the average eccentricity of a path $P \cup Q$, we have
$$
ecc (G (p + 1, q - 1)) - ecc (G (p, q)) \geq |G| > 0.
$$
Since $G$ is a nontrivial graph with at least one vertex, we have strict inequality. \medskip

\noindent This completes the proof. \qed

Chemical trees (trees with maximum vertex degree at most four) provide the graph representations of
alkanes \cite{GuPo86}. It is therefore a natural problem to study trees with bounded maximum
degree. The path $P_n$ is the unique tree with $\Delta = 2$, while the star $S_n$ is the unique
tree with $\Delta = n-1$. Therefore, we can assume that $3 \leq \Delta \leq n - 2$.

The broom $B (n, \Delta)$ is a tree consisting of a star $S_{\Delta + 1}$ and a path of length $n -
\Delta - 1$ attached to an arbitrary pendent vertex of the star (see Figure 1). It is proven that
among trees with maximum vertex degree equal to $\Delta$, the broom $B (n, \Delta)$ uniquely
minimizes the Estrada index \cite{IlSt10}, the largest eigenvalue of the adjacency matrix
\cite{LiGu07}, distance spectral radius \cite{StIl10}, etc.

\begin{figure}[ht]
  \center
  \includegraphics [width = 7cm]{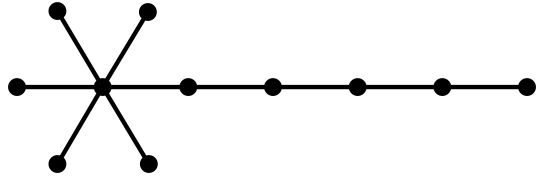}
  \caption { \textit{ The broom $B (11, 6)$. } }
\end{figure}

\begin{theorem}
\label{thm-broom} Let $T \not \cong B (n, \Delta)$ be an arbitrary tree on $n$ vertices with
maximum vertex degree $\Delta$. Then
$$
ecc (B (n, \Delta)) > ecc (T).
$$
\end{theorem}

\proof Fix a vertex $v$ of degree $\Delta$ as a root and let $T_1, T_2, \ldots, T_{\Delta}$ be the
trees attached at $v$. We can repeatedly apply the transformation described in Theorem \ref{thm-pi}
at any vertex of degree at least three with largest eccentricity from the root in every tree $T_i$,
as long as $T_i$ does not become a path. When all trees $T_{1}, T_{2},\dots, T_{\Delta}$ turn into
paths, we can again apply transformation from Theorem~\ref{thm-pi} at the vertex~$v$ as long as
there exist at least two paths of length greater than one, further decreasing the average
eccentricity. Finally, we arrive at the broom $B (n, \Delta)$ as the unique tree with maximum
average eccentricity. \qed

By direct verification, it holds
$$
ecc (B (n, \Delta)) = \frac{1}{n} \left ( \left \lfloor \frac{(n - \Delta + 2)(3 (n - \Delta + 2) -
2)}{4} \right \rfloor + (n - \Delta + 1) (\Delta - 2) \right).
$$

If $\Delta>2$, we can apply the transformation from Theorem~\ref{thm-pi} at the vertex of
degree~$\Delta$ in $B (n, \Delta)$ and obtain $B (n, \Delta-1)$. Thus, we have the following chain
of inequalities
$$
ecc (S_{n}) = ecc (B (n, n - 1)) < ecc (B (n, n - 2)) < \cdots < ecc(B (n, 3))< ecc(B
(n,2))=ecc(P_{n}).
$$

Also, it follows that $B (n, 3)$ has the second maximum average eccentricity among trees on $n$
vertices. On the other hand, the addition of an arbitrary edge in $G$ cannot decrease the average
eccentricity and clearly $\varepsilon (v) \geq 1$ with equality if and only if $deg (v) = n - 1$.

\begin{theorem}
Among graphs on $n$ vertices, the path $P_n$ attains the maximum average eccentricity index, while
the complete graph $K_n$ attains the minimum average eccentricity index.
\end{theorem}

A starlike tree is a tree with exactly one vertex of degree at least 3. We denote by $S
(n_{1},n_{2},\ldots,n_{k})$ the starlike tree  of order $n$ having a branching vertex $v$ and
$$
S (n_{1},n_{2},\ldots,n_{k})-v=P_{n_1}\cup P_{n_2}\cup \ldots \cup P_{n_k},
$$
where $n_1\geq n_2\geq \ldots\geq n_k \geq 1$. Clearly, the numbers $n_1, n_2, \ldots, n_k$
determine the starlike tree up to isomorphism and $n = n_1 + n_2 + \ldots + n_k + 1$. The starlike
tree $BS (n, k) \cong S (n_{1},n_{2},\ldots,n_{k})$ is balanced if all paths have almost equal
lengths, i.e., $|n_i - n_j| \leqslant 1$ for every $1 \leqslant i < j \leqslant k$.

Let $x = (x_1, x_2, \ldots, x_n)$ and $y = (y_1, y_2, \ldots, y_n)$ be two integer arrays of
length~$n$. We say that $x$ majorizes $y$ and write $x \succ y$ if the elements of these arrays
satisfy following conditions:
\begin{enumerate}
\renewcommand{\labelenumi}{(\roman{enumi})}

\item $x_1 \geqslant x_2 \geqslant \ldots \geqslant x_n$ and $y_1 \geqslant y_2 \geqslant \ldots \geqslant
y_n$,
\item $x_1 + x_2 + \ldots + x_k \geqslant y_1 + y_2 + \ldots + y_k$,
for every $1 \leqslant k < n$,
\item $x_1 + x_2 + \ldots + x_n = y_1 + y_2 + \ldots + y_n$.
\end{enumerate}

\begin{theorem}
Let $p=(p_1, p_2, \ldots, p_k)$ and $q=(q_1, q_1, \ldots, q_k)$ be two arrays of length $k
\geqslant 2$, such that $p \succ q$ and $n - 1 = p_1 + p_2 + \ldots + p_k = q_1 + q_2 + \ldots
q_k$. Then
\begin{equation}
\label{eq:starlike} ecc (S (p_1, p_2, \ldots, p_k)) \geq ecc (S (q_1, q_2, \ldots, q_k)).
\end{equation}
\end{theorem}

\begin{proof}
We will proceed by induction on the size of the array $k$. For $k = 2$, we can directly apply
transformation from Theorem~\ref{thm-pi} on tree $S (q_1, q_2)$ several times, in order to get $S
(p_1, p_2)$. Assume that the inequality (\ref{eq:starlike}) holds for all lengths less than or
equal to~$k$. If there exist an index $1 \leqslant m < k$ such that $p_1 + p_2 + \ldots + p_m = q_1
+ q_2 + \ldots + q_m$, we can apply the induction hypothesis  on two parts $S (q_1, q_2, \ldots,
q_m) \cup S (q_{m + 1}, q_{m + 2}, \ldots, q_k)$ and get $S (p_1, p_2, \ldots, p_m) \cup S (p_{m +
1}, p_{m + 2}, \ldots, p_k)$. Otherwise, we have strict inequalities $p_1 + p_2 + \ldots + p_m >
q_1 + q_2 + \ldots + q_m$ for all indices $1 \leqslant m < k$. We can transform tree $S (q_1, q_2,
\ldots, q_k)$ into $S (q_1 + 1, q_2, \ldots, q_{r-1}, q_r - 1, q_{r + 1}, \ldots, q_k)$, where $r$
is the largest index such that $q_r \neq 1$. The condition $p \succ q$ is preserved, and we can
continue until the array $q$ transforms into $p$, while at every step we increase the average
eccentricity.
\end{proof}

\begin{corollary}\label{cor:order}
Let $T = S (n_1, n_2, \ldots, n_k) $ be a starlike tree with $n$ vertices and $k$ pendent paths.
Then
$$
ecc (B (n, k)) \geq ecc (T) \geq ecc (BS (n, k)) .
$$
The left equality holds if and only if $T \cong B (n, k)$ and the right equality holds if and only
if $T \cong BS (n, k)$.
\end{corollary}

\begin{definition}
Let $u v$ be a bridge of the graph $G$ and let $H$ and $H'$ be the non-trivial components of $G$,
such that $u \in H$ and $v \in H'$. Construct the graph $G'$ by identifying the vertices $u$ and
$v$ (and call this vertex also $u'$) with additional pendent edge $u'v'$. We say that $G' = \sigma
(G, uv)$ is a $\sigma$-transform of $G$.
\end{definition}

\begin{theorem}
Let $G' = \sigma (G, uv)$ be a $\sigma$-transform of $G$. Then,
$$
ecc (G') < ecc (G).
$$
\end{theorem}

\proof Let $x$ be a vertex on the maximum distance from $u$ in the graph $H$ and let $y$ be a
vertex on the maximum distance from $v$ in the graph $H'$. Without loss of generality assume that
$d (u, x) \geq d (v, y)$. It can be easily seen that for arbitrary vertex $w \in G$ different than
$v$ and $y$ holds $\varepsilon_{G} (w) \geq \varepsilon_{G'} (w)$. For the vertex $y$ we have
$\varepsilon_G (y) = d (y, v) + 1 + d (u, x) > d (y, u') + d (u', x) = \varepsilon_{G'} (y)$. For
the vertex $v$ we have $\varepsilon_G (v) = 1 + d (u, x) = 1 + d (u', x) = \varepsilon_{G'} (v')$.
Finally, we have strict inequality $\sum_{w \in G}\varepsilon (w) > \sum_{w \in G'}\varepsilon
(w')$ and the result follows. \qed

Using previous theorem, one can easily prove that the star $S_n$ is the unique tree with minimal
value of the average eccentricity $ecc (S_n) = 2 - \frac{1}{n}$ among trees with $n$ vertices.
Furthermore, by repeated use of $\sigma$ transformation, the graph $S_n'$ (obtained from a star
$S_n$ with additional edge connecting two pendent vertices) has minimal value of the average
eccentricity $ecc (S_n') = 2 - \frac{1}{n}$ among unicyclic graphs with $n$ vertices. It follows by
simple analyze of the average eccentricity of extremal unicyclic graphs obtained from a triangle
$C_3$ with $a$, $b$ and $c$ pendent vertices attached to the vertices of a triangle, with $a + b +
c + 3 = n$.

\section{Conjecture regarding the independence number}

A set of vertices $S$ in a graph $G$ is independent if no neighbor of a vertex of $S$ belongs to
$S$. The independence number $\alpha = \alpha (G)$ is the maximum cardinality of an independent set
of $G$.

\begin{conjecture}[A.478-U]
\label{con-1} For every $n \geq 4$ it holds
$$
\alpha (G) + ecc (G) \leq \left\{
\begin{array}{l l}
  \frac{3n^2-2n-1}{4n} + \frac{n+1}{2} & \quad \mbox{if $n$ is odd}\\
  \frac{3n^2-4n-4}{4n} + \frac{n+2}{2} & \quad \mbox{if $n$ is even}
\end{array} \right.,
$$
with equality if and only if $G \cong P_n$ for odd $n$ and $G \cong B (n, 3)$ for even $n$.
\end{conjecture}

Clearly, the sum $\alpha (G) + ecc (G)$ is maximized for some tree. Let $T^*$ be the extremal tree
and let $P = v_0 v_1 \ldots v_d$ be a diametrical path of an extremal tree $T^*$. The maximum
possible independence number of this tree is $\lceil \frac{d + 1}{2} \rceil + n - d - 1$.

\begin{lemma}
\label{lem:lepa} Let $T$ be an arbitrary tree on $n$ vertices, not isomorphic to a path $P_n$. Then
there is a pendent vertex $v$ such that for each $u \in T$ it holds
$$
\varepsilon_{T} (u) = \varepsilon_{T - v} (u).
$$
\end{lemma}

\proof Let $P = \{ p_1, p_2, \ldots, p_k\}$ be the set of all pendent vertices of $T$. Construct a
directed graph $D$ with the vertex set $P$ and a directed edge from $p_i$ to $p_j$ if the vertex
$p_j$ is the unique most distant vertex from $p_i$ in the tree $T$. This way we get a directed
graph $D$ with $k$ vertices and at most $k$ edges. If there is a vertex $p_i$ with indegree 0,
after deleting the vertex $p_i$ the eccentricities of all other pendent vertices remain the same.
Otherwise, the number of edges must be equal to $k$ and suppose that the indegrees of all vertices
$p_1, p_2, \ldots, p_k$ are greater than or equal to 1. It simply follows that in this case the
graph $D$ is composed of directed cycles of length $\geq 2$. Let $C = c_1 c_2 \ldots c_s$ be the
cycle of length $s \geq 3$ in $D$. The directed edges $c_1 \rightarrow c_2$ and $c_2 \rightarrow
c_3$ imply that the distance $d (c_1, c_2)$ is strictly less than the distance $d (c_2, c_3)$. By
extending this argument, we have
$$
d (c_1, c_2) < d (c_2, c_3) < \ldots < d (c_{s - 1}, c_s) < d (c_s, c_1),
$$
and the vertex $c_s$ is more distant from $c_1$ than the vertex $c_2$, a contradiction.

Therefore, all cycles from $D$ have exactly two vertices and for each vertex $p_i$ there is a
unique eccentric vertex $f (p_i)$ such that $d (p_i, f (p_i)) = \varepsilon (v_i) = \varepsilon (f
(v_i))$. Consider the unique path in $T$ from $p_i$ to $f (p_i)$. This path must contain a central
vertex. Otherwise one can construct strictly longer path by going from $p_i$ to a central vertex
$c$ and then to some other pendent vertex $p_j$ on distance $r$ or $r - 1$ from $c$, where $r$
denotes the radius of a tree. It follows that paths from $p_i$ to $f (p_i)$ and from $p_j$ and $f
(p_j)$ have a vertex in common (if $T$ is bicentral, the longest path from any pendent vertex must
contain both central vertices). This implies that for $p_i, f (p_i), p_j, f (p_j)$ it can not hold
that all vertices most distance from them are unique. Therefore, there is a pendent vertex $v$ such
that for all other pendent vertices $p_i$, we have $\varepsilon_{T} (p_i) = \varepsilon_{T - v}
(p_i)$.

Now, let $u$ be an arbitrary non-pendent vertex of $T$ and assume that $v$ is the unique pendent
vertex such that $\varepsilon (u) = d (u, v)$. After deleting the vertex $u$ from $T$, the tree
decomposes into connected components (at least two). Consider an arbitrary component $C$ that does
not contain vertex $v$ and one pendent vertex $p_i$ from this component. It follows that $d (v,
p_i) = d (v, u) + d (u, p_i)$, and $v$ is the most distant vertex from $p_i$, since the distances
from $u$ to any other vertex from the subtree $C$ are strictly less than $d (u, p_i) + d (v, u)$.
This is a contradiction and $d (v, u) < \varepsilon (u)$ for each $u \in V$. Finally, for each $u
\in T$ it holds $\varepsilon_{T} (u) = \varepsilon_{T - v} (u)$.\qed

By finding a pendent vertex from Lemma \ref{lem:lepa} and reattaching it to $v_1$ or $v_{d - 1}$,
we do not increase the value of $\alpha (G) + ecc (G)$, while keeping the diameter the same. It
follows that the broom tree $B (n, n - d + 2)$ has the same value $\alpha (G) + ecc (G)$ as the
extremal tree $T^*$. By direct calculation we have
\begin{eqnarray*}
ecc (B (n, \Delta)) + \alpha (B (n, \Delta)) &=& \frac{1}{n} \left ( \left \lfloor \frac{(n -
\Delta + 2)(3 (n - \Delta + 2) - 2)}{4} \right \rfloor
+ (n - \Delta + 1) (\Delta - 2) \right) \\
&& + \left \lceil \frac{n-\Delta+2}{2} \right \rceil + (\Delta - 2) \\
&=& \left\{
\begin{array}{l l}
  \frac{5n}{4} - \frac{\Delta (\Delta - 2)}{4n}-\frac{1}{2} & \quad \mbox{if $n - \Delta$ is even }\\
  \frac{5n}{4} - \frac{\Delta (\Delta - 2)}{4n}-\frac{1}{4n} & \quad \mbox{if $n - \Delta$ is odd }\\
\end{array} \right..
\end{eqnarray*}

For $\Delta = 2$ and $\Delta = 3$, we have
$$
ecc (B (n, 2)) + \alpha (B (n, 2)) = \left\{
\begin{array}{l l}
  \frac{5n}{4} - \frac{1}{2} & \quad \mbox{if $n$ is even }\\
  \frac{5n}{4} - \frac{1}{4n} & \quad \mbox{if $n$ is odd }\\
\end{array} \right.
$$
$$
ecc (B (n, 3)) + \alpha (B (n, 3)) = \left\{
\begin{array}{l l}
  \frac{5n}{4} - \frac{3}{4n}-\frac{1}{2} & \quad \mbox{if $n$ is odd }\\
  \frac{5n}{4} - \frac{3}{4n}-\frac{1}{4n} & \quad \mbox{if $n$ is even }\\
\end{array} \right.
$$
It follows that for $n \geq 3$ the maximum value of $ecc (G) + \alpha (G)$ is achieved uniquely for
$B (n, 2) \cong P_n$ if $n$ is odd, and for $B (n, 3)$ if $n$ is even. This completes the proof of
Conjecture~\ref{con-1}.

\begin{remark}
Actually the extremal trees are double brooms $D (d, a, b)$, obtained from the path $P_{d + 1}$ by
attaching $a$ endvertices to one end and $b$ endvertices to the other of the path $P_{d+1}$. The
double broom has diameter $d$, order $n = d+a+b+1$ and the same average eccentricity as the broom
$B (n, n - d)$. The authors in \cite{DaGoSw04} showed that the extremal graph with the maximum
average eccentricity for given order $n$ and radius $r$ is any double broom of diameter $2r$.
\end{remark}

\section{Conjectures regarding the Randi\' c index}

In 1975, the chemist Milan Randi\' c \cite{Ra79} proposed a topological index $Ra (G)$ under the
name 'branching index', suitable for measuring the extent of branching of the carbon-atom skeleton
of saturated hydrocarbons. Randi\' c index of a graph $G$ is defined as
$$
Ra (G) = \sum_{uv \in E} \frac{1}{\sqrt{deg (v) \cdot deg (u)}}.
$$
Later, Bollob\' as and Erd\" os \cite{BoEr98} generalized this index by replacing the exponent
$-\frac{1}{2}$ with any real number $\alpha$, which is called the general Randi\' c index. For a
comprehensive survey of its mathematical properties, see \cite{LiSh08} and the book of Li and
Gutman \cite{LiGu06}. For example, it holds
$$
Ra (P_n) = \frac{n - 3 + 2 \sqrt{2}}{2} \qquad Ra (S_n) = \sqrt{n - 1} \qquad Ra (K_n) =
\frac{n}{2}.
$$

\begin{conjecture}[A.462-U]
For every $n \geq 4$ it holds
$$
Ra (G) + ecc (G) \leq \left\{
\begin{array}{l l}
  \frac{n-3+2\sqrt{2}}{2} + \frac{3n+1}{4} \cdot \frac{n-1}{n} & \quad \mbox{if $n$ is odd}\\
  \frac{n-3+2\sqrt{2}}{2} + \frac{3n-2}{4} & \quad \mbox{if $n$ is even}
\end{array} \right.,
$$
with equality if and only if $G \cong P_n$.
\end{conjecture}

\begin{conjecture}[A.464-U]
For every $n \geq 4$ it holds
$$
Ra (G) \cdot ecc (G) \leq \left\{
\begin{array}{l l}
  \frac{n-3+2\sqrt{2}}{2} \cdot \frac{3n+1}{4} \cdot \frac{n-1}{n} & \quad \mbox{if $n$ is odd}\\
  \frac{n-3+2\sqrt{2}}{2} \cdot \frac{3n-2}{4} & \quad \mbox{if $n$ is even}
\end{array} \right.,
$$
with equality if and only if $G \cong P_n$.
\end{conjecture}

We will use one classical result from the theory of Randi\' c index.

\begin{theorem} \cite{CaGuHaPa03,PaGu01}
Among all graphs of order $n$, regular graphs attain the maximum Randi\' c index $\frac{n}{2}$.
\end{theorem}

Since the second maximum value of the average eccentricity index among connected graphs on $n$
vertices is achieved for the broom $B (n, 3)$, we have the following upper bound
$$
ecc (B (n, 3)) = \frac{(n - 1) \cdot ecc (P_{n - 1}) + (n - 2)}{n} \leq ecc (P_n) - \frac{1}{2} -
\frac{1}{2n},
$$
with equality if and only if $n$ is odd.

For $G \not \cong P_n$, it follows
\begin{eqnarray*}
Ra (G) + ecc (G) &\leq& \frac{n}{2} + ecc (B (n, 3)) \\
&=& \frac{n}{2} + ecc (P_n) - \frac{1}{2} - \frac{1}{2n} \\
&\leq& Ra (P_n) + ecc (P_n) + \frac{3-2\sqrt{2}}{2} - \frac{n +
1}{2n} \\
&<& Ra (P_n) + ecc (P_n),
\end{eqnarray*}
since $\frac{n+1}{2n} > \frac{1}{2} > \frac{3 - 2 \sqrt{2}}{2}$ holds for all $n \geq 1$.

For the second conjecture, similarly we have
\begin{eqnarray*}
Ra (G) \cdot ecc (G) &\leq& \frac{n}{2} \cdot ecc (B (n, 3)) \\
&=& \left(Ra (P_n) + \frac{3-2\sqrt{2}}{2} \right) \left (ecc (P_n) - \frac{1}{2} - \frac{1}{2n} \right) \\
&=& Ra (P_n) \cdot ecc (P_n) + \frac{3-2\sqrt{2}}{2} \cdot ecc (P_n) - \frac{n+1}{4} \\
&<& Ra (P_n) \cdot ecc (P_n),
\end{eqnarray*}
since $ecc (P_n) < \frac{3n}{4} - \frac{1}{4} < \frac{n+1}{2 (3 - 2 \sqrt{2})}$.

This completes the proof of both conjectures.

\section{Conjecture regarding the clique number}

The clique number of a graph $G$ is the size of a maximal complete subgraph of $G$ and it is
denoted as $\omega (G)$.

The lollipop graph $LP (n, k)$ is obtained from a complete graph $K_k$ and a path $P_{n - k + 1}$,
by joining one of the end vertices of $P_{n - k + 1}$ to one vertex of $K_k$ (see Figure 2). An
asymptotically sharp upper bound for the eccentric connectivity index is derived independently in
\cite{DoSaVu10} and \cite{MoMuSw10}, with the extremal graph $LP (n, \lfloor n / 3 \rfloor)$.
Furthermore, it is shown that the eccentric connectivity index grows no faster than a cubic
polynomial in the number of vertices.

\begin{figure}[ht]
  \center
  \includegraphics [width = 8cm]{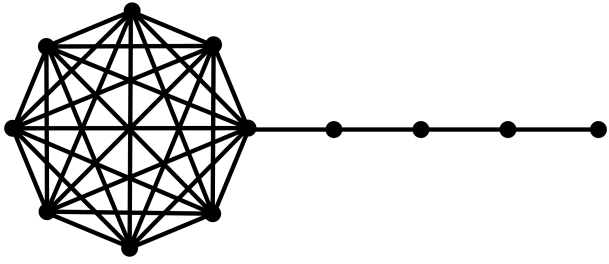}
  \caption { \textit{ The lollipop graph $LP (12, 8)$. } }
\end{figure}

\begin{conjecture} [A.488-U]
For every $n \geq 4$ the maximum value of $ecc (G) \cdot \omega (G)$ is achieved for some lollipop
graph.
\end{conjecture}

Let $C$ be an arbitrary clique of size $k$. Since the removal of the edges potentially increases
$ecc (G)$, we can assume that trees are attached to the vertices of $C$. Then by applying Theorem
\ref{thm-pi}, we get the graph composed of the clique $C$ and pendent paths attached to the
vertices of $C$. Using the transformation similar to $G (p, q) \mapsto G (p + 1, q - 1)$ where we
increase the length of the longest path attached to $C$, it follows that the extremal graph is
exactly $LP (n, k)$. Since $ecc (LP (n, k)) = ecc (B (n, k))$, we have
\begin{eqnarray*}
ecc (LP (n, k)) \cdot \omega (LP (n, k)) &=& \frac{1}{n} \cdot \Big ( (n - k - 2) ecc (P_{n - k -
2}) + (n - k + 1)(k - 2) \Big ) \cdot k \\
&=& \frac{k}{n} \cdot \left \lfloor \frac{(-k^2 - 2 k (-1 + n) + n (2 + 3 n))}{4} \right \rfloor.
\end{eqnarray*}

Let $f(x) = x\left(-x^2+2 x -2 x n+ 2 n+3 n^2\right)$ and $f' (x) = -3 x^2-4 x (n-1)+n (3 n + 2)$.
By simple analysis for $x \in [1, n]$, it follows that the function $f (x)$ achieves the maximum
value exactly for the larger root of the equation $f' (x) = 0$. Therefore, the maximum value of
$ecc (G) \cdot \omega (G)$ is achieved for integers closest to
$$
k^* = \frac{1}{3} \left(2-2 n+\sqrt{4-2 n+13 n^2}\right).
$$

\section{Conjecture regarding the minimum vertex degree}

A matching in a graph $G$ is a set of edges in which no two edges are adjacent. A vertex is matched
(or saturated) if it is incident to an edge in the matching; otherwise the vertex is unmatched. A
perfect matching (or 1-factor) is a matching which matches all vertices of the graph.


\begin{conjecture}[A.100-U]
\label{con-6} For every $n \geq 4$ it holds
$$
\delta (G) \cdot ecc (G) \leq \left\{
\begin{array}{l l}
  2n - 2 & \quad \mbox{if $n$ is even }\\
  (n - 2)(2 - \frac{1}{2}) & \quad \mbox{if $n$ is odd }\\
\end{array} \right.,
$$
with equality if and only if $G \cong K_n \setminus M$, where $M$ is a perfect matching if $n$ is
even, or a perfect matching on $n-1$ vertices with an additional edge between the non-saturated
vertex and another vertex if $n$ is odd.
\end{conjecture}

Let $K_n \setminus \{uv\}$ be the graph obtained from a complete graph $K_n$ by deleting the edge
$uv$. Define the almost-path-clique graph $PC (k, \delta)$ from a path $P_k$ by replacing each
vertex of degree 2 by the graph $K_{\delta + 1} \setminus \{u_i v_i\}$, $i = 2, 3, \ldots, k - 1$
and replacing pendent vertices by the graphs $K_{\delta + 2} \setminus \{u_1 v_1\}$ and $K_{\delta
+ 2} \setminus \{u_k v_k\}$. Furthermore, for each $i = 1, 2, \ldots, k - 1$ the vertices $u_i$ and
$v_{i + 1}$ are adjacent (see Figure 3).

\begin{figure}[ht]
  \center
  \includegraphics [width = 14cm]{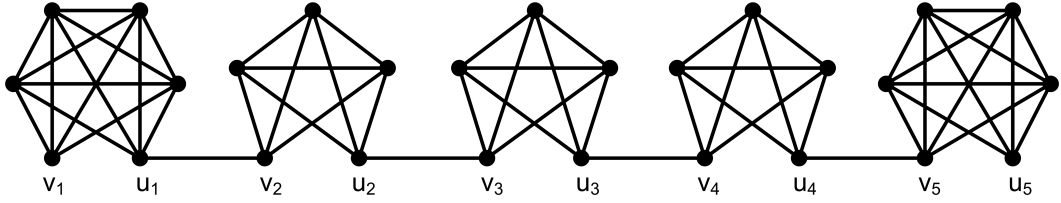}
  \caption { \textit{ The graph $PC (5, 4)$ with 27 vertices.} }
\end{figure}

The graph $PC (k, \delta)$ has $n = k (\delta + 1) + 2$ vertices and minimum vertex degree
$\delta$. Assume that $k$ is an even number. For each $i = 1, 2, \ldots, \frac{k}{2}$, we have the
following contributions of the vertices in $K_{\delta + 1} \setminus \{u_i v_i\}$:
\begin{itemize}
\item the vertex $u_i$ has eccentricity $\frac{3k}{2} + 3 (\frac{k}{2} - i) = 3k - 3i$,
\item the vertex $v_i$ has eccentricity $\frac{3k}{2} + 2 + 3 (\frac{k}{2} - i) = 3k - 3i + 2$,
\item the remaining $\delta - 1$ or $\delta$ vertices have eccentricity $\frac{3k}{2} + 1 + 3 (\frac{k}{2} -
i) = 3k - 3i + 1$.
\end{itemize}

Finally, the average eccentricity of the graph $PC (k, \delta)$ is equal to
\begin{eqnarray*}
ecc (PC (k, \delta)) &=& \frac{2}{n} \cdot \left( 3k - 2 + \sum_{i = 1}^{k/2} (3k - 3i) + (3k -3i
+ 2) + (\delta - 1)(3k - 3i + 1) \right) \\
&=& \frac{1}{k (\delta + 1) + 2} \cdot \left ( \frac{9 \delta k^2}{4} + \frac{9k^2}{4}+ \frac{11
k}{2}-\frac{\delta k}{2} -4 \right) \\
&=& \frac{9k}{4} - \frac{1}{2} + \frac{3 (k - 2)}{2 (k \delta + k + 2)}.
\end{eqnarray*}

The product of the average eccentricity and the minimum vertex degree is equal to
$$
ecc (PC (k, \delta)) \cdot \delta (PC (k, \delta)) = \frac{9k\delta}{4} - \frac{\delta}{2} +
\frac{3 \delta (k - 2)}{2 (k \delta + k + 2)}.
$$

For each $k \geq \delta \geq 10$ we have the following inequality
$$
\frac{9k \delta}{4} - \frac{\delta}{2} > 2 (k \delta + k + 2) - 4,
$$
which is equivalent with
$$
k \delta - 8k - 2 \delta = k (\delta - 8) - 2 \delta > 0.
$$

This refutes Conjecture \ref{con-6}, and one can easily construct similar counterexamples for odd
$k$ or $n$ not of the form $k (\delta + 1) + 2$. Note that this construction is very similar to the
one described in \cite{DaGoSw04}, but derived independently.

\section{Concluding remarks}

In this paper we studied the mathematical properties of the average eccentricity $ecc (G)$ of a
connected graph $G$, which is deeply connected with the eccentric connectivity index. We resolved
or refuted five conjectures on the average eccentricity and other graph invariants -- clique
number, Randi\' c index, independence number and minimum vertex degree.

We conclude the paper by restating some other conjectures dealing with the average eccentricity.
All conjectures were generated by AGX system \cite{Ao06} and we also verified them on the set of
all graphs with $\leq 10$ vertices and trees with $\leq 20$ vertices (with the help of
Nauty~\cite{Nauty} for the generation of non-isomorphic graphs).

\begin{conjecture}[A.462-L]
For every $n \geq 4$ it holds
$$
Ra (G) + ecc (G) \geq \sqrt{n - 1} + 2 - \frac{1}{n},
$$
with equality if and only if $G \cong S_n$.
\end{conjecture}

\begin{conjecture}[A.464-L]
For every $n \geq 4$ it holds
$$
Ra (G) \cdot ecc (G) \geq \left\{
\begin{array}{l l}
  \frac{n}{2} & \quad \mbox{if $n \leq 13$ }\\
  \sqrt{n - 1} \cdot \left (2 - \frac{1}{n} \right) & \quad \mbox{if $n > 13$ }\\
\end{array} \right.,
$$
with equality if and only if $G \cong K_n$ for $n \leq 13$ or $G \cong S_n$ for $n > 13$.
\end{conjecture}

\begin{conjecture}[A.458-L]
For every $n \geq 4$ it holds
$$
\lambda (G) + ecc (G) \geq \sqrt{n - 1} + \left(2 - \frac{1}{n} \right),
$$
with equality if and only if $G \cong S_n$, where $\lambda (G)$ is the largest eigenvalue of the
adjacency matrix of $G$.
\end{conjecture}

\begin{conjecture}[A.460-L]
For every $n \geq 4$ it holds
$$
\lambda (G) \cdot ecc (G) \geq \sqrt{n - 1} \cdot \left ( 2 - \frac{1}{n} \right),
$$
with equality if and only if $G \cong S_n$.
\end{conjecture}

\begin{conjecture} [A.479-U]
For every $n \geq 4$ the maximum value of $ecc (G) / \alpha (G)$ is achieved for some graph $G$
composed of two cliques linked by a path.
\end{conjecture}

\begin{conjecture} [A.492-U]
For every $n \geq 4$ the maximum value of $ecc (G) \cdot \chi (G)$ is achieved for some lollipop
graph, where $\chi (G)$ denotes the chromatic number of $G$.
\end{conjecture}

A dominating set of a graph $G$ is a subset $D$ of $V$ such that every vertex not in $D$ is joined
to at least one member of $D$ by some edge. The domination number $\gamma (G)$ is the number of
vertices in a smallest dominating set for $G$ \cite{HaHeSl98}.

\begin{conjecture}[A.464-L]
For every $n \geq 4$ it holds
$$
\gamma (G) + ecc (G) \geq \left\{
\begin{array}{l l}
  \lfloor \frac{n + 1}{3} \rfloor + \frac{(3n+1)n}{4(n-1)} & \quad \mbox{if $n$ is odd and $n \not \equiv 1 \pmod 3$ }\\
  \lfloor \frac{n + 1}{3} \rfloor + \frac{3n-2}{4} & \quad \mbox{if $n$ is even and $n \not \equiv 1 \pmod 3$ }\\
  \frac{13n-16}{12} - \frac{3}{4n} & \quad \mbox{if $n$ is odd and $n \equiv 1 \pmod 3$ }\\
  \frac{13n-16}{12} - \frac{1}{n} & \quad \mbox{if $n$ is even and $n \equiv 1 \pmod 3$ }
\end{array} \right.,
$$
with equality if and only if $G \cong P_n$ for $n \not \equiv 1 \pmod 3$ or $G$ is a tree with $D =
n - 2$ and $\gamma = \lfloor \frac{n+1}{3} \rfloor$ for $n \equiv 1 \pmod 3$.
\end{conjecture}

We tested this conjecture and derived the following corrected version

\begin{conjecture}[A.464-L]
For every $n \geq 4$ it holds
$$
\gamma (G) + ecc (G) \geq \left\{
\begin{array}{l l}
  \lceil \frac{n}{3} \rceil + \frac{1}{n} \left \lfloor \frac{3}{4}n^2 - \frac{1}{2}n \right \rfloor & \quad \mbox{if $n \not \equiv 0 \pmod 3$ }\\
  \frac{n}{3} + 2 - \frac{3}{n} + \frac{1}{n} \left \lfloor \frac{3}{4}(n-1)^2 - \frac{1}{2}(n-1) \right \rfloor & \quad \mbox{if $n \equiv 0 \pmod 3$ }\\
\end{array} \right.,
$$
with equality if and only if $G \cong P_n$ for $n \not \equiv 0 \pmod 3$ or $G \cong D_n$ for $n
\equiv 0 \pmod 3$, where $D_n \cong S (n - 4, 2, 1)$ is a tree obtained from a path $P_{n-1} = v_1
v_2 \ldots v_{n-1}$ by attaching a pendent vertex to $v_3$.
\end{conjecture}

Similarly as for the independence number, the extremal graphs are trees. The domination number of a
path $P_n$ is $\lceil \frac{n}{3} \rceil$, and since the path has maximum average eccentricity in
order to prove the conjecture one has to consider trees with $\lceil \frac{n}{3} \rceil < \gamma
\leq \lfloor \frac{n}{2} \rfloor$.\medskip

It would be also interesting to determine extremal regular (cubic) graphs with respect to the
average eccentricity, or to study some other derivative indices (such as eccentric distance sum
\cite{YuFeIl10}, or augmented and super augmented eccentric connectivity indices~\cite{DuGuMa08}).

\bigskip {\bf Acknowledgement. } This work was supported by Research
Grant 144007 of Serbian Ministry of Science and Technological Development.


\begin{thebibliography}{99}

\bibitem{Ao06}
    M. Aouchiche, \textit{Comparaison Automatisée d'Invariants en Th\' eorie des Graphes},
    PhD Thesis, \' Ecole Polytechnique de Montr\' eal, February 2006.


\bibitem{AGX}
    M. Aouchiche, J. M. Bonnefoy, A. Fidahoussen, G. Caporossi, P. Hansen, L. Hiesse, J. Lacher\' e,
    A. Monhait, \textit{Variable Neighborhood Search for Extremal Graphs, 14. The AutoGraphiX 2 System},
    in Global Optimization: from Theory to Implementation, L. Liberti, N. Maculan (eds.), Springer,
    281--310, 2006.

\bibitem{AuHa07}
    M. Aouchiche, P. Hansen,
    \textit{Automated results and conjectures on average distance in graphs},
    in: Graph Theory in Paris, Trends Math. VI (2007) 21--36.

\bibitem{AoBrHa08}
    M. Aouchiche, G. Brinkmann, P. Hansen,
    \textit{Variable neighborhood search for extremal graphs. 21. Conjectures and results about the independence
    number},
    Discrete Appl. Math. {\bf 156} (2008) 2530--2542.

\bibitem{AuHa10}
    M. Aouchiche, P. Hansen,
    \textit{Nordhaus-Gaddum relations for proximity and remoteness in graphs},
    Comp. Math. Appl. {\bf 59} (2010) 2827--2835.

\bibitem{AuHa10a}
    M. Aouchiche, P. Hansen,
    \textit{A survey of automated conjectures in spectral graph theory},
    Linear Algebra Appl. {\bf 432} (2010) 2293--2322.

\bibitem{BoEr98}
    B. Bollob\' as, P. Erd\" os,
    \textit{Graphs of extremal weights},
    Ars Combin. {\bf 50} (1998) 225--233.

\bibitem{BuHa90}
    F. Buckley, F. Harary, \textit{Distance in Graphs},
    Addison--Wesley, Redwood City, California, 1990.

\bibitem{CaHa00}
    G. Caporossi, P. Hansen,
    \textit{Variable neighborhood search for extremal graphs. I. The AutoGraphiX system},
    Discrete Math. {\bf 212} (2000) 29--44.

\bibitem{CaGuHaPa03}
    G. Caporossi, I. Gutman, P. Hansen, L. Pavlovi\' c,
    \textit{Graphs with maximum connectivity index},
    Comput. Biol. Chem. {\bf 27} (2003) 85--90.

\bibitem{DaGoSw04}
    P. Dankelmann, W. Goddard, C. S. Swart,
    \textit{The average eccentricity of a graph and its subgraphs},
    Util. Math. {\bf 65} (2004) 41--51.

\bibitem{DaEn00}
    P. Dankelmann, R. Entringer,
    \textit{Average distance, minimum distance, and spanning trees},
    J. Graph Theory {\bf 33} (2000) 1--13.

\bibitem{DoEnGu01}
    A. A. Dobrynin, R. C. Entringer, I. Gutman,
    \textit{Wiener index of trees: theory and applications},
    Acta Appl. Math. {\bf 66} (2001) 211--249.

\bibitem{DoSa09}
    T. Do\v sli\' c, M. Saheli,
    \textit{Eccentric connectivity index of composite graphs}, manuscript, 2009.

\bibitem{DoSaVu10}
    T. Do\v sli\' c, M. Saheli, D. Vuki\v cevi\' c,
    \textit{Eccentric connectivity index: Extremal graphs and values}, Iranian J. Math. Chem.
    {\bf 1} (2010) 000--000.

\bibitem{DuGuMa08}
    H. Dureja, S. Gupta, A. K. Madan,
    \textit{Predicting anti-HIV-1 activity of 6-arylbenzonitriles: Computational approach using
    superaugmented eccentric connectivity topochemical indices},
    J. Mol. Graph. Model. {\bf 26} (2008) 1020--1029.

\bibitem{GuPo86}
    I. Gutman, O. E. Polansky, {\it Mathematical Concepts in Organic Chemistry},
    Springer--Verlag, Berlin, 1986.

\bibitem{HaMl01}
    P. Hansen, N. Mladenovi\' c,
    \textit{Variable neighborhood search: Principles and applications},
    Eur. J. Oper. Res. {\bf 130} (2001) 449--467.

\bibitem{HaVu09}
    P. Hansen, D. Vuki\v cevi\' c,
    \textit{Variable neighborhood search for extremal graphs. 23. On the Randi\' c index and the chromatic
    number},
    Discrete Math. {\bf 309} (2009) 4228--4234.

\bibitem{HaMlMo10}
  P. Hansen, N. Mladenovi\' c, J. A. Moreno P\'{e}rez,
  \textit{Variable neighborhood search: algorithms and applications}.
  Annals Oper. Res. {\bf 175} (2010) 367--407.

\bibitem{HaHeSl98}
    T. W. Haynes, S. Hedetniemi, P. Slater,
    \textit{Fundamentals of Domination in Graphs},
    Marcel Dekker, New York, 1998.

\bibitem{Il10}
    A. Ili\' c,
    \textit{Eccentric connectivity index},
    in: I. Gutman, B. Furtula, Novel Molecular Structure Descriptors -- Theory and Applications II,
    Mathemtical Chemistry Monographs, Volume 9, University of Kragujevac, 2010.

\bibitem{IlKlMi10}
    A. Ili\' c, S. Klav\v zar, M. Milanovi\' c,
    \textit{On distance balanced graphs},
    European J. Combin. {\bf 31} (2010) 733--737.

\bibitem{IlSt10}
    A. Ili\' c, D. Stevanovi\' c,
    \textit{The Estrada index of chemical trees},
    J. Math. Chem. {\bf 47} (2010) 305--314.

\bibitem{IlGu10}
    A. Ili\' c, I. Gutman,
    \textit{Eccentric connectivity index of chemical trees},
    MATCH Commun. Math. Comput. Chem. {\bf 65} (2011) 731--744.

\bibitem{KhYoAsWa09}
    M. K. Khalifeh, H. Yousefi-Azari, A. R. Ashrafi, S. G. Wagner,
    \textit{Some new results on distance-based graph invariants},
    European J. Combin. {\bf 30} (2009) 1149--1163.

\bibitem{LiSh08}
    X. Li, Y. T. Shi,
    \textit{A Survey on the Randi\' c Index},
    MATCH Commun. Math. Comput. Chem. {\bf 59} (2008) 127--156.

\bibitem{LiGu06}
    X. Li, I. Gutman,
    \textit{Mathematical Aspects of Randi\' c-Type Molecular Structure Descriptors},
    Mathemtical Chemistry Monographs, Volume 1, University of Kragujevac, 2006.

\bibitem{LiGu07}
    W. Lin, X. Guo,
    \textit{On the largest eigenvalues of trees with perfect matchings},
    J. Math. Chem. {\bf 42} (2007) 1057--1067.


\bibitem{Nauty}
    B. McKay, \textit{Nauty}, \url{http://cs.anu.edu.au/~bdm/nauty/}.

\bibitem{MoMuSw10}
    M. J. Morgan, S. Mukwembi, H. C. Swart,
    \textit{On the eccentric connectivity index of a graph},
    Discrete Math., doi:10.1016/j.disc.2009.12.013.

\bibitem{Ra79}
    M. Randi\' c,
    \textit{On characterization of molecular branching},
    J. Amer. Chem. Soc. {\bf 97} (1975) 6609--6615.

\bibitem{PaGu01}
    L. Pavlovi\' c, I. Gutman,
    \textit{Graphs with extremal connectivity index},
    Novi. Sad. J. Math. {\bf 31} (2001) 53--58.

\bibitem{SeVuHa07}
    J. Sedlar, D. Vuki\v cevi\' c, P. Hansen,
    \textit{Using Size for Bounding Expressions of Graph Invariants},
    G-2007-100 manuscript, 2007.

\bibitem{SeVuAoHa07}
    J. Sedlar, D. Vuki\v cevi\' c, M. Aouchiche, P. Hansen,
    \textit{Variable Neighborhood Search for Extremal Graphs. 25. Products of Connectivity and Distance Measures},
    G-2007-47 manuscript, 2007.

\bibitem{ShGoMa97}
    V. Sharma, R. Goswami, A. K. Madan,
    \textit{Eccentric connectivity index: A novel highly discriminating topological descriptor
    for structure--property and structure--activity studies},
    J. Chem. Inf. Comput. Sci. {\bf 37} (1997) 273--282.

\bibitem{StIl10}
    D. Stevanovi\' c, A. Ili\' c,
    \textit{Distance spectral radius of trees with fixed maximum degree},
    Electron. J. Linear Algebra {\bf 20} (2010) 168--179.

\bibitem{YuFeIl10}
    G. Yu, L. Feng, A. Ili\' c,
    \textit{On the eccentric distance sum of trees and unicyclic graphs},
    J. Math. Anal. Appl. {\bf 375} (2011) 99--107.

\bibitem{ZhDu09}
    B. Zhou, Z. Du,
    \textit{On eccentric connectivity index},
    MATCH Commun. Math. Comput. Chem. {\bf 63} (2010) 181--198.


\end{thebibliography}
\end{document}